\documentclass[11pt, twoside, leqno]{article}

\usepackage{amssymb}
\usepackage{amsmath}
\usepackage{amsthm}
\usepackage{color}
\usepackage{mathrsfs}

\allowdisplaybreaks

\def\rr{{\mathbb R}}
\def\rn{{{\rr}^n}}
\def\zz{{\mathbb Z}}

\def\nn{{\mathbb N}}

\def\cd{{\mathcal D}}

\def\cs{{\mathcal S}}

\def\mi{{\mathrm I}}

\def\fz{\infty}
\def\az{\alpha}
\def\bz{\beta}

\def\gz{{\gamma}}

\def\lz{\lambda}

\def\oz{{\omega}}

\def\vz{\varphi}

\def\lf{\left}
\def\r{\right}
\def\lfz{{\lfloor}}
\def\rfz{{\rfloor}}
\def\rf{\rfloor}
\def\la{\langle}
\def\ra{\rangle}
\def\hs{\hspace{0.25cm}}
\def\ls{\lesssim}

\def\noz{\nonumber}
\def\wz{\widetilde}

\def\st{\subset}
\def\com{\complement}
\def\bh{\backslash}

\def\dyt{\,\frac{dy\,dt}{t}}
\def\dt{\,\frac{dt}{t}}

\def\supp{\mathop\mathrm{\,supp\,}}

\def\loc{\mathop\mathrm{\,loc\,}}

\def\esup{\mathop\mathrm{\,esssup\,}}

\def\aa{{\mathbb A}}
\def\lv{{L^{\vz}(\rn)}}

\def\hv{{H^{\vz}(\rn)}}
\def\bmo{{\mathop\mathrm {BMO}}}

\newtheorem{thm}{Theorem}[section]
\newtheorem{prop}[thm]{Proposition}
\newtheorem{lem}[thm]{Lemma}
\newtheorem{cor}[thm]{Corollary}

\theoremstyle{definition}
\newtheorem{defn}[thm]{Definition}

\numberwithin{equation}{section}

\begin{document}

\arraycolsep=1pt

\title{\bf\Large New Real-Variable Characterizations of Musielak-Orlicz Hardy Spaces
\footnotetext{\hspace{-0.35cm} 2010 {\it
Mathematics Subject Classification}. Primary 42B25; Secondary 42B30, 42B35.
\endgraf {\it Key words and phrases}. Musielak-Orlicz function, Hardy space, atom,
maximal function, Littlewood-Paley $g$-function, Littlewood-Paley $g_\lambda^\ast$-function.
\endgraf Yiyu Liang and Dachun Yang are partially supported by 2010 Joint Research Project Between
China Scholarship Council and German Academic Exchange
Service (PPP) (Grant No. LiuJinOu [2010]6066). Jizheng Huang is supported by the National Natural
Science Foundation (Grant No. 11001002) of China.
Dachun Yang is also partially supported by the National
Natural Science Foundation (Grant No. 11171027) of China
and Program for Changjiang Scholars and Innovative
Research Team in University of China. Part of this paper was finished
during the course of the visit of Yiyu Liang and Dachun Yang to the Mathematisches Institut
of Friedrich-Schiller-Universit\"at Jena and they want to express their sincere
and deep thanks for the gracious hospitality of the the Research group
``Function spaces" therein.} }
\author{Yiyu Liang, Jizheng Huang and Dachun Yang\,\footnote{Corresponding author}}
\date{}

\maketitle


\begin{center}
\begin{minipage}{13.8cm}
{\small {\bf Abstract}\quad Let $\varphi: {\mathbb R^n}\times
[0,\infty)\to[0,\infty)$ be such that
$\varphi(x,\cdot)$ is an Orlicz function and $\varphi(\cdot,t)$ is
a Muckenhoupt $A_\infty({\mathbb R^n})$ weight. The Musielak-Orlicz Hardy
space $H^{\varphi}(\mathbb R^n)$ is defined to be the space
of all $f\in{\mathcal S}'({\mathbb R^n})$ such that the
grand maximal function $f^*$ belongs to the Musielak-Orlicz
space $L^\varphi(\mathbb R^n)$. Luong Dang Ky established its atomic
characterization. In this paper, the authors establish some new real-variable characterizations
of $H^{\varphi}(\mathbb R^n)$ in terms of the vertical or the non-tangential maximal
functions, or the Littlewood-Paley $g$-function or $g_\lambda^\ast$-function, via first establishing
a Musielak-Orlicz Fefferman-Stein vector-valued
inequality. Moreover, the range of $\lambda$ in the
$g_\lambda^\ast$-function characterization of $H^\varphi(\mathbb R^n)$
coincides with the known best results, when $H^\varphi(\mathbb R^n)$
is the classical Hardy space $H^p(\mathbb R^n)$, with $p\in (0,1]$,
or its weighted variant.}
\end{minipage}
\end{center}


\vspace{0.2cm}

\section{Introduction\label{s1}}

\hskip\parindent As the generalization of $L^p(\mathbb R^n)$, the
Orlicz space was introduced by Birnbaum-Orlicz in \cite{bo31}
and Orlicz in \cite{o32}. Since then, the theory of the Orlicz
spaces themselves has been well developed and these spaces have
been widely used in probability, statistics, potential theory,
partial differential equations, as well as harmonic analysis and
some other fields of analysis; see, for example, \cite{aikm00, io02,
mw08}. Moreover, Orlicz-Hardy spaces are also suitable
substitutes of Orlicz spaces in dealing with many problems of
analysis; see, for example, \cite{j80, s79, v87, jyz09,jy10,jy11}. Recall
that Orlicz-Hardy spaces and their dual spaces were first studied by
Str\"{o}mberg \cite{s79} and Janson \cite{j80} on
$\mathbb R^n$ and Viviani \cite{v87} on spaces of homogeneous
type in the sense of Coifman and Weiss \cite{cw71}.

Recently, Ky \cite{ky} introduced a new
\emph{Musielak-Orlicz Hardy space}, $\hv$, via the grand maximal
function, which generalizes both the Orlicz-Hardy space of
Str\"omberg \cite{s79} and Janson \cite{j80} and the weighted
Hardy space $H^p_\oz(\rn)$ with $\oz\in A_{\fz}(\rn)$ studied by
Garc\'\i a-Cuerva \cite{gr85} and Str\"omberg and Torchinsky
\cite{st89}, here and in what follows,
$A_q(\rn)$ with $q\in[1,\fz]$ denotes the
\emph{class of Muckenhoupt's weights} (see, for example,
\cite{g79,gr85} for their definitions and properties) and
we always assume that $\vz$ is a \emph{growth function},
which means that $\vz:\,\rn\times[0,\fz)\to[0,\fz)$ is a Musielak-Orlicz function such
that $\vz(x,\cdot)$ is an Orlicz function and $\vz(\cdot,t)$ is
a Muckenhoupt $A_\fz(\rn)$ weight. Musielak-Orlicz functions are the
natural generalization of Orlicz functions that may vary in the
spatial variables; see, for example, \cite{d05,dhr09,ky,m83}.
Recall that the motivation to study function spaces of
Musielak-Orlicz type comes from applications to elasticity, fluid
dynamics, image processing, nonlinear partial differential
equations and the calculus of variation; see, for example,
\cite{bg10,bgk,bijz07,d05,dhr09,ky}.

In \cite{ky}, Ky established the atomic characterization of $\hv$
and, moreover, Ky \cite{ky} further introduced the $\mathop\mathrm{BMO}$-type space
$\mathop\mathrm{BMO}_{\vz}(\rn)$, which was proved to be the dual
space of $\hv$; as an interesting application, Ky proved
that the class of pointwise multipliers for $\bmo(\rn)$,
characterized by Nakai and Yabuta \cite{ny85}, is the dual space
of $L^1(\rn)+H^{\log}(\rn)$, where $H^{\log}(\rn)$ is the Musielak-Orlicz Hardy
space related to the growth function
$$\vz(x,t):=\frac{t}{\log(e+|x|)+\log(e+t)}$$
for all $x\in\rn$ and $t\in [0,\fz)$. Furthermore, the Lusin
area function and the molecular characterizations of $\hv$ were obtained in
\cite{hyy11}. As an application of the Lusin
area function characterization of $\hv$, the $\varphi$-Carleson measure characterization
of $\mathop\mathrm{BMO}_{\varphi}(\mathbb{R}^n)$
was also given in \cite{hyy11}. It is worth noticing that
some special Musielak-Orlicz Hardy spaces appear naturally
in the study of the products of functions in $\bmo(\rn)$ and
$H^1(\rn)$ (see \cite{bgk,bijz07}), and the endpoint estimates for
the div-curl lemma and the commutators of singular integral operators
(see \cite{bfg10,bgk,ky2}).
Moreover, the local Musielak-Orlicz
Hardy space, $h^\vz(\rn)$, and its dual space, ${\rm bmo}_\vz(\rn)$,
were studied in \cite{dcyy} and
some applications of $h^\vz(\rn)$ and ${\rm bmo}_\vz(\rn)$,
to pointwise multipliers of BMO-type spaces and to
the boundedness of local Riesz transforms and pseudo-differential
operators on $h^\vz(\rn)$,
were also obtained in \cite{dcyy}.

In this paper, we establish some new real-variable characterizations
of $H^{\varphi}(\mathbb R^n)$ in terms of the vertical or the non-tangential maximal
functions, and
in terms of the Littlewood-Paley $g$-function or $g_\lambda^\ast$-function, via first establishing
a Musielak-Orlicz Fefferman-Stein vector-valued
inequality. Moreover, the range of $\lambda$ in the
$g_\lambda^\ast$-function characterization of $H^\varphi(\mathbb R^n)$
coincides with the known best results, when $H^\varphi(\mathbb R^n)$
is the classical Hardy space $H^p(\mathbb R^n)$, with $p\in (0,1]$,
or its weighted variant.

To be precise, this paper is organized as follows.

In Section \ref{s2}, we recall some
notions concerning growth functions and some
of their properties established in \cite{ky}.
Then via some skillful applications of these properties on growth functions,
such as their equivalent property  that
$$\vz(x,t)\sim\int_0^t\frac{\vz(x,s)}s\,ds\quad\text{for all}
\quad (x,t)\in \mathbb R^n\times[0,\infty)$$
(see Lemma \ref{lem1}(ii) below),
we establish an interpolation theorem of Musielak-Orlicz type
(see Theorem \ref{inter} below) and also a vector-valued version
(see Theorem \ref{inter-vector} below). As a corollary,
we immediately obtain a Musielak-Orlicz Fefferman-Stein vector-valued
inequality (see Theorem \ref{f-s} below), which plays a key role in establishing
the $g$-function characterization of $\hv$ in Section \ref{s4} below
and might also be very useful in some other applications including the further study
of function spaces of Musielak-Orlicz type, for example, Besov-type and Triebel-Lizorkin-type spaces.

Section \ref{s3} is devoted to establishing some maximal function
characterizations of $\hv$ in terms of the vertical and the non-tangential maximal
functions (see Theorem \ref{max-ch} below), via first obtaining
some key inequalities (see Theorem \ref{t3.1} below) involving
the grand, the vertical and the tangential Peetre-type maximal functions.

In Section \ref{s4}, by using the Lusin area function characterization of $H_\vz(\rn)$
established in \cite{hyy11}, we obtain the $g$-function characterization of $\hv$
(see Theorem \ref{g-ch} below).
To do so, except using the Musielak-Orlicz Fefferman-Stein vector-valued
inequality established in Theorem \ref{f-s} of this paper,
we also need to invoke the discrete Calder\'on reproducing
formula obtained by Lu and Zhu \cite[Theorem 2.1]{lz11}
and some key estimates from \cite[Lemmas 2.1 and 2.2]{lz11}
(see also the estimates \eqref{4.y1} and \eqref{4.16} below).
Moreover, by borrowing some ideas from Folland and Stein \cite{fs82}
and Aguilera and Segovia \cite{as77},
we further obtain the Littlewood-Paley $g_\lz^*$-function
characterization of $\hv$ for all $\lz\in(2q/p,\fz)$ (see Theorem \ref{gast-ch} below).
We point out that even when $\vz(x,t):=t^p$ for all $x\in\rn$ and $t\in(0,\fz)$,
or $\vz(x,t):=w(x)t^p$ for all $x\in\rn$ and $t\in(0,\fz)$,
with $p\in(0,1]$, $q\in[1,\fz)$ and $w\in A_q(\rn)$,
the range of $\lz$ is the known best possible; see, respectively,
\cite[p.\,221,\,Corollary (7.4)]{fs82} and \cite[Theorem 2]{as77}.
In this sense, the range of $\lz$ in Theorem \ref{gast-ch}
might also be the best possible.

We remark that the Littlewood-Paley function
characterizations of $H^\varphi(\mathbb{R}^n)$ have local variants,
which will be studied in a forthcoming paper; see \cite{dcyy}
for the definition of local Musielak-Orlicz Hardy spaces $h^\varphi(\mathbb{R}^n)$.

Finally we make some conventions on notation. Throughout the whole
paper, we denote by $C$ a \emph{positive constant} which is
independent of the main parameters, but it may vary from line to
line. We also use $C_{(\az,\bz,\cdots)}$ to denote a \emph{positive
constant} depending on the indicated parameters $\gz$, $\bz$,
$\cdots$. The {\it symbol} $A\ls B$ means that $A\le CB$. If $A\ls
B$ and $B\ls A$, then we write $A\sim B$. The {\it symbol} $\lfz
s\rfz$ for $s\in\rr$ denotes the maximal integer not more than $s$.
For any measurable subset $E$ of $\rn$, we denote by $E^\complement$ the {\it set}
$\rn\setminus E$ and by $\chi_{E}$ its \emph{characteristic function}.
For any cube $Q\st\rn$, we use $\ell(Q)$ to denote its side length.
We also set $\nn:=\{1,\,2,\,
\cdots\}$ and $\zz_+:=\nn\cup\{0\}$.
For any given function $g$ on $\rn$, if $\int_\rn g(x)\,dx\neq 0$,
we let $L_g:= -1$; otherwise, we let $L_g\in\zz_+$ be the {\it maximal integer} such that
$g$ has vanishing moments up to order $L_g$, namely,
$\int_{\rn}g(x)x^{\az}\,dx=0$ for all multi-indices $\az$ with
$|\az|\le L_g$.


\section{Preliminaries\label{s2}}

\hskip\parindent In Subsection \ref{s2.1}, we first recall some
notions concerning growth functions and some
of their properties established in \cite{ky}.
Then in Subsection \ref{s2.2} we establish an interpolation theorem of Musielak-Orlicz type
and also a vector-valued version. As a corollary,
we obtain a Musielak-Orlicz type Fefferman-Stein vector-valued
inequality. In Subsection \ref{s2.3}, we recall the notion of
Musielak-Orlicz Hardy spaces and some of their known properties.

\subsection{Growth functions\label{s2.1}}

\hskip\parindent Recall that a function
$\Phi:[0,\fz)\to[0,\fz)$ is called an \emph{Orlicz function} if it
is nondecreasing, $\Phi(0)=0$, $\Phi(t)>0$ for $t\in(0,\fz)$ and
$\lim_{t\to\fz}\Phi(t)=\fz$ (see, for example,
\cite{m83,rr91,rr02}). The function $\Phi$ is said to be of
\emph{upper type $p$} (resp. \emph{lower type $p$}) for some $p\in[0,\fz)$, if
there exists a positive constant $C$ such that for all
$t\in[1,\fz)$ (resp. $t\in[0,1]$) and $s\in[0,\fz)$,
$\Phi(st)\le Ct^p \Phi(s).$

For a given function $\vz:\,\rn\times[0,\fz)\to[0,\fz)$ such that for
any $x\in\rn$, $\vz(x,\cdot)$ is an Orlicz function,
$\vz$ is called to be of \emph{uniformly upper type $p$} (resp.
\emph{uniformly lower type $p$}) for some $p\in[0,\fz)$ if there
exists a positive constant $C$ such that for all $x\in\rn$,
$t\in[0,\fz)$ and $s\in[1,\fz)$ (resp. $s\in[0,1]$), $\vz(x,st)\le Cs^p\vz(x,t)$.
Let
\begin{equation}\label{2.1}
i(\vz):=\sup\{p\in(0,\fz):\ \vz\ \text{is of uniformly lower
type}\ p\}.
\end{equation}
Observe that $i(\vz)$ may not be attainable, namely, $\vz$ may not be of uniformly
lower type $i(\vz)$; see below for some examples.

Let $\vz:\rn\times[0,\fz)\to[0,\fz)$ satisfy that
$x\mapsto\vz(x,t)$ is measurable for all $t\in[0,\fz)$. Following
\cite{ky}, $\vz(\cdot,t)$ is called \emph{uniformly locally
integrable} if, for all compact sets $K$ in $\rn$,
$$\int_{K}\sup_{t\in(0,\fz)}\lf\{|\vz(x,t)|
\lf[\int_{K}|\vz(y,t)|\,dy\r]^{-1}\r\}\,dx<\fz.$$

\begin{defn}\label{d2.1}
Let $\vz:\rn\times[0,\fz)\to[0,\fz)$ be uniformly locally
integrable. The function $\vz(\cdot,t)$ is said to satisfy the
\emph{uniformly Muckenhoupt condition for some $q\in[1,\fz)$},
denoted by $\vz\in\aa_q(\rn)$, if, when $q\in (1,\fz)$,
\begin{equation}\label{2.2}
\aa_q (\vz):=\sup_{t\in
[0,\fz)}\sup_{B\subset\rn}\frac{1}{|B|^q}\int_B
\vz(x,t)\,dx \lf\{\int_B
[\vz(y,t)]^{-q'/q}\,dy\r\}^{q/q'}<\fz,
\end{equation}
where $1/q+1/q'=1$, or
\begin{equation*}
\aa_1 (\vz):=\sup_{t\in [0,\fz)}
\sup_{B\subset\rn}\frac{1}{|B|}\int_B \vz(x,t)\,dx
\lf(\esup_{y\in B}[\vz(y,t)]^{-1}\r)<\fz.
\end{equation*}
Here the first supremums are taken over all $t\in[0,\fz)$ and the
second ones over all balls $B\subset\rn$.
\end{defn}

Recall that $\aa_q(\rn)$ with $q\in[1,\fz)$ in Definition
\ref{d2.1} was introduced by Ky \cite{ky}.
We have the following properties for $\aa_q(\rn)$ with $q\in [1,\fz)$, whose proofs
are similar to those in \cite{gr85,g09}.

\begin{lem}\label{l2.x1}
$\mathrm{(i)}$ $\aa_1(\rn)\subset\aa_p(\rn)\subset\aa_q(\rn)$ for
$1\le p\le q<\fz$.

$\mathrm{(ii)}$ If $\vz\in\aa_p(\rn)$ with $p\in(1,\fz)$, then
there exists $q\in(1,p)$ such that $\vz\in\aa_q(\rn)$.
%
\end{lem}

Let $\aa_{\fz}(\rn):=\cup_{q\in[1,\fz)}\aa_{q}(\rn)$
and define the \emph{critical index}, $q(\vz)$,
of $\vz\in\aa_{\fz}(\rn)$ by
\begin{equation}\label{2.3}
q(\vz):=\inf\lf\{q\in[1,\fz):\ \vz\in\aa_{q}(\rn)\r\}.
\end{equation}
By Lemma \ref{l2.x1}(ii), we see that if $q(\vz)\in(1,\fz)$,
then $\vz\not\in\aa_{q(\vz)}(\rn)$. Moreover, there exists
$\vz\not\in\aa_1(\rn)$ such that $q(\vz)=1$ (see, for example, \cite{jn87}).

Now we introduce the notion of growth functions.

\begin{defn}\label{d2.2}
A function $ \varphi:\mathbb R^n\times[0,\infty)\to[0,\infty)$
is called a \emph{growth function} if
the following conditions are satisfied:
\vspace{-0.25cm}
\begin{enumerate}
\item[(i)] $\vz$ is a \emph{Musielak-Orlicz function}, namely,
\vspace{-0.2cm}
\begin{enumerate}
    \item[(i)$_1$] the function $\vz(x,\cdot):[0,\fz)\to[0,\fz)$ is an
    Orlicz function for all $x\in\rn$;
    \vspace{-0.2cm}
    \item [(i)$_2$] the function $\vz(\cdot,t)$ is a measurable
    function for all $t\in[0,\fz)$.
\end{enumerate}
\vspace{-0.25cm} \item[(ii)] $\vz\in \aa_{\fz}(\rn)$.
\vspace{-0.25cm} \item[(iii)] $\vz$ is of positive
uniformly lower type $p$ for some $p\in(0,1]$ and of uniformly
upper type 1.
\end{enumerate}
\end{defn}

Clearly, $\vz(x,t):=\oz(x)\Phi(t)$ is a growth function if
$\oz\in A_{\fz}(\rn)$ and $\Phi$ is an Orlicz function of lower
type $p$ for some $p\in(0,1]$ and of upper type 1. It is known
that, for $p\in(0,1]$, if $\Phi(t):=t^p$ for all $t\in [0,\fz)$,
then $\Phi$ is an Orlicz function of lower type $p$ and of upper type $p$;
for $p\in[\frac{1}{2},1]$, if
$\Phi(t):= t^p/\ln(e+t)$ for all $t\in [0,\fz)$, then $\Phi$ is an
Orlicz function of lower type $q$ for $q\in(0, p)$ and of upper type $p$; for
$p\in(0,\frac{1}{2}]$, if $\Phi(t):=t^p\ln(e+t)$ for all $t\in
[0,\fz)$, then $\Phi$ is an Orlicz function of lower type $p$ and of upper type $q$
for $q\in(p,1]$. Recall that if an Orlicz function is of upper
type $p\in(0,1)$, then it is also of upper type 1. Another typical and useful growth
function is
$$\vz(x,t):=\frac{t^{\az}}{[\ln(e+|x|)]^{\bz}+[\ln(e+t)]^{\gz}}$$
for all $x\in\rn$ and $t\in[0,\fz)$, with any $\az\in(0,1]$,
$\bz\in[0,\fz)$ and $\gz\in [0,2\az(1+\ln2)]$; more precisely,
$\vz\in \aa_1(\rn)$, $\vz$ is of uniformly upper type $\az$ and
$i(\vz)=\az$ which is not attainable (see \cite{ky}).

Throughout the whole paper, we \emph{always
assume that $\vz$ is a growth function} as in Definition
\ref{d2.2}. Let us now introduce the Musielak-Orlicz space.

The \emph{Musielak-Orlicz space $L^{\vz}(\rn)$} is defined to be the space
of all measurable functions $f$ such that
$\int_{\rn}\vz(x,|f(x)|)\,dx<\fz$ with \emph{Luxembourg
norm}
$$\|f\|_{L^{\vz}(\rn)}:=\inf\lf\{\lz\in(0,\fz):\ \int_{\rn}
\vz\lf(x,\frac{|f(x)|}{\lz}\r)\,dx\le1\r\}.$$
In what follows, for any measurable subset $E$ of $\rn$, we denote
$\int_E\vz(x,t)\,dx$ by the \emph{symbol $\vz(E,t)$} for any $t\in[0,\fz)$.

The following Lemmas \ref{lem1}, \ref{lem2} and \ref{lem3} on the properties of growth functions are,
respectively, \cite[Lemmas 4.1, 4.2 and 4.3]{ky}.

\begin{lem}\label{lem1}
{\rm(i)} Let $\vz$ be a growth function. Then $\vz$ is uniformly
$\sigma$-quasi-subadditive on $\rn\times[0,\fz)$, namely, there
exists a positive constant $C$ such that for all
$(x,t_j)\in\rn\times[0,\fz)$ with $j\in\nn$,
$$\vz\lf(x,\sum_{j=1}^{\fz}t_j\r)\le C\sum_{j=1}^{\fz}\vz(x,t_j).$$

{\rm(ii)} Let $\vz$ be a growth function and
$$\wz{\vz}(x,t):=\int_0^t\frac{\vz(x,s)}{s}\,ds\quad {\rm for\ all}\ \
(x,t)\in\rn\times[0,\fz).$$
Then $\wz{\vz}$ is a growth function,
which is equivalent to $\vz$; moreover, $\wz{\vz}(x,\cdot)$ is
continuous and strictly increasing.
\end{lem}

\begin{lem}\label{lem2}
Let $\varphi$ be a growth function. Then

{\rm (i)} For all $f\in
L^\varphi(\mathbb R^n)\setminus\{0\}$,
$$\int_{\mathbb R^n}\varphi\lf(x,\frac{|f(x)|}{\|f\|_{L^\varphi(\rn)}}\r)=1.$$

{\rm (ii)} $\lim_{k\to\infty}\|f_k\|_{L^\varphi(\rn)}=0$
if and only if
$$\lim_{k\to\infty}\int_{\mathbb
R^n}\varphi(x,|f_k(x)|)dx=0.$$
\end{lem}

\begin{lem}\label{lem3}
For a given positive constant $\wz C$, there exists a positive constant
$C$ such that the following hold:

{\rm (i)} The inequality
$$\int_{\mathbb R^n}\varphi\lf(x,\frac{|f(x)|}{\lambda}\r)dx\le \wz C\quad \text{for}\quad \lambda\in (0,\fz)$$
implies that $ \|f\|_{L^\varphi(\rn)}\le C\lambda.$

{\rm (ii)} The inequality
$$\sum_j \varphi\lf(Q_j,\frac{t_j}{\lambda}\r)\le \wz C\quad \text{for}\quad \lambda\in (0,\fz)$$
implies that
$$
\inf\lf\{\alpha>0:\ \sum_j\varphi\lf(Q_j,\frac{t_j}{\alpha}\r)\le1\r\}\le
C\lambda,
$$
where $\{t_j\}_j$ is a sequence of positive constants and
$\{Q_j\}_j$ a sequence of cubes.
\end{lem}

\subsection{The Musielak-Orlicz Fefferman-Stein vector-valued
inequality\label{s2.2}}

\hskip\parindent
In this subsection, we establish an interpolation theorem of
operators, in the spirit of the Marcinkiewicz
interpolation theorem,
associated with a growth function,
which may have independent interest.
In what follows, for any nonnegative locally integrable function
$w$ on $\rn$ and $p\in (0,\fz)$, the \emph{space $L^p_w(\rn)$} is defined
to be the space of all measurable functions $f$ such that
$$\|f\|_{L^p_w(\rn)}:=\lf\{\int_\rn|f(x)|^pw(x)\,dx\r\}^{1/p}<\fz.$$

\begin{thm}\label{inter}
Let $p_{1},p_2\in (0,\fz)$, $p_{1}<p_{2}$ and $\vz$ be a
Musielak-Orlicz function with uniformly lower type $p_\vz^-$ and
uniformly upper type $p_\vz^+$.
If $0<p_{1}<p_\vz^-\le p_\vz^+<p_{2}<\fz$
and $T$ is a sublinear
operator defined on $L^{p_1}_{\vz(\cdot,1)}(\rn)+L^{p_2}_{\vz(\cdot,1)}(\rn)$
satisfying that for $i\in\{1,2\}$, all $\az\in(0,\fz)$ and $t\in(0,\fz)$,
\begin{eqnarray}\label{2.x1}
\vz(\{x\in\rn:\ |Tf(x)|>\az\},t)\le C_i\az^{-p_i}\int_\rn|f(x)|^{p_i}\vz(x,t)\,dx,
\end{eqnarray}
where $C_i$ is a positive constant independent of $f$,
$t$ and $\az$. Then $T$ is bounded on
$L^\vz(\rn)$ and, moreover, there exists a positive constant $C$
such that for all $f\in L^\vz(\rn)$,
\begin{eqnarray*}
\int_\rn\vz(x,|Tf(x)|)\,dx\le C\int_\rn\vz(x,|f(x)|)\,dx.
\end{eqnarray*}
\end{thm}

\begin{proof} First observe that for all $t\in(0,\fz)$,
$$\int_\rn|f(x)|^p\vz(x,t)\,dx<\fz\quad \text{if and only if}\quad
\int_\rn|f(x)|^p\vz(x,1)\,dx<\fz.$$
Thus, the spaces $L^p_{\vz(\cdot,t)}(\rn)$ and $L^p_{\vz(\cdot,1)}(\rn)$
coincide as sets. Now we show that
$L^\vz(\rn)\st L^{p_{1}}_{\vz(\cdot,1)}(\rn)+L^{p_{2}}_{\vz(\cdot,1)}(\rn)$.

For any given $t\in (0,\fz)$, we decompose $f\in \lv$ as
$$f=f\chi_{\{x\in\rn:\ |f(x)|>t\}}+f\chi_{\{x\in\rn:\ |f(x)|\le t\}}
=: f^t+f_t.$$
Then by the fact that $\vz$ is of uniformly lower type $p_\vz^-$ and $p_1<p_\vz^-$,
we conclude that
\begin{eqnarray*}
\int_{\rn}|f^{ t}(x)|^{p_{1}}\vz(x,1)\,dx
&&\ls\int_{\{x\in\rn:\ |f(x)|> t\}}|f(x)|^{p_{1}}
\lf[\frac t{|f(x)|}\r]^{p_\vz^-}\vz\lf(x,\frac{|f(x)|} t\r)\,dx\\
&&\ls t^{p_{1}}\int_\rn\vz\lf(x,\frac{|f(x)|} t\r)\,dx<\fz,
\end{eqnarray*}
namely, $f^{ t}\in L^{p_{1}}_{\vz(\cdot,1)}(\rn)$.
Similarly we have $f_{ t}\in L^{p_{2}}_{\vz(\cdot,1)}(\rn)$ and hence
$Tf$ is well defined.

By the fact that $T$ is sublinear
and Lemma \ref{lem1}(ii), we further see that
\begin{eqnarray*}
\int_{\rn}\vz(x,|Tf(x)|)\,dx &&
\sim\int_{0}^{\fz}\frac1t\int_{\{x\in\rn:\ |Tf(x)|>t\}}\vz(x,t)\,dx\,dt\\
&&\ls\int_{0}^{\fz}\frac1t\int_{\{x\in\rn:\ |Tf^t(x)|>t/2\}}\vz(x,t)\,dx\,dt\\
&&\hs+\int_{0}^{\fz}\frac1t\int_{\{x\in\rn:\ |Tf_t(x)|>t/2\}}\cdots=:\mi_1+\mi_2.
\end{eqnarray*}

On $\mi_1$, since $T$ is of weak type $(p_{1},p_{1})$ (namely,
\eqref{2.x1} with $i=1$),
$\vz$ is of uniformly lower type $p_\vz^-$ and $p_1<p_\vz^-$,
we conclude that
\begin{eqnarray*}
\mi_1
&&\ls\int_{0}^{\fz}\frac1t\lf(\frac t2\r)^{-p_1}\int_\rn |f^t(x)|^{p_1}\vz(x,t)\,dx\,dt\\
&&\sim\int_{0}^{\fz}\frac1{t^{1+p_1}}\int_{\{x\in\rn:\ |f(x)|>t\}} |f(x)|^{p_1}\vz(x,t)\,dx\,dt\\
&&\sim\int_{0}^{\fz}\frac1{t^{1+p_1}}\int_{\{x\in\rn:\ |f(x)|>t\}}\vz(x,t)
\lf[\int_t^{|f(x)|} p_1s^{p_1-1}\,ds+t^{p_1}\r]\,dx\,dt\\
&&\sim\int_{0}^{\fz}s^{p_1-1}\int_{\{x\in\rn:\ |f(x)|> s\}}
\int_0^s \frac{\vz(x,t)}{t^{1+p_1}}\,dt\,dx\,ds
+\int_{0}^{\fz}\frac1{t}\int_{\{x\in\rn:\ |f(x)|>t\}}\vz(x,t)\,dx\,dt\\
&&\ls\int_{0}^{\fz}s^{p_1-1}\int_{\{x\in\rn:\ |f(x)|> s\}}
\vz(x,s)s^{-p_\vz^-}\int_0^s \frac{1}{t^{1+p_1-p_\vz^-}}\,dt\,dx\,ds
+\int_\rn\vz(x,|f(x)|)\,dx\\
&&\sim\int_{0}^{\fz}\frac1s\int_{\{x\in\rn:\ |f(x)|> s\}}\vz(x,s)dx\,ds
+\int_\rn\vz(x,|f(x)|)\,dx
\sim\int_\rn\vz(x,|f(x)|)\,dx.
\end{eqnarray*}

Also, from the weak type $(p_{2},p_{2})$ of $T$
(namely, \eqref{2.x1} with $i=2$), the uniformly upper type $p_\vz^+$
property of $\vz$ and $p_\vz^+<p_2$,
we deduce that
\begin{eqnarray*}
\mi_2
&&\ls\int_{0}^{\fz}\frac1t\lf(\frac t2\r)^{-p_2}\int_\rn |f_t(x)|^{p_2}\vz(x,t)\,dx\,dt\\
&&\sim\int_{0}^{\fz}\frac1{t^{1+p_2}}\int_{\{x\in\rn:\ |f(x)|\le t\}} |f(x)|^{p_2}\vz(x,t)\,dx\,dt\\
&&\sim\int_{0}^{\fz}\frac1{t^{1+p_2}}\int_{\{x\in\rn:\ |f(x)|\le t\}}\vz(x,t)
\int_0^{|f(x)|} p_2s^{p_2-1}\,ds\,dx\,dt\\
&&\sim\int_{0}^{\fz}s^{p_2-1}\int_{\{x\in\rn:\ |f(x)|> s\}}
\int_s^\fz \frac{\vz(x,t)}{t^{1+p_2}}\,dt\,dx\,ds\\
&&\ls\int_{0}^{\fz}s^{p_2-1}\int_{\{x\in\rn:\ |f(x)|> s\}}
\vz(x,s)s^{-p_\vz^+}\int_s^\fz \frac{1}{t^{1+p_2-p_\vz^+}}\,dt\,dx\,ds\\
&&\sim\int_{0}^{\fz}\frac1s\int_{\{x\in\rn:\ |f(x)|> s\}}\vz(x,s)dx\,ds
\sim\int_\rn\vz(x,|f(x)|)\,dx.
\end{eqnarray*}
Thus, $T$ is bounded on $\lv$, which completes the proof of Theorem \ref{inter}.
\end{proof}

Recall that for any locally integrable function $f$ and $x\in\rn$,
the \emph{Hardy-Littlewood maximal function $Mf(x)$} is defined by
$$Mf(x):=\sup_{x\in B}\frac1{|B|}\int_B|f(y)|\,dy,$$
where the supremum is taken over all balls $B$ containing $x$.
Let $q(\vz)$ be as in \eqref{2.3}.
As a simple corollary of Theorem \ref{inter}, together with
the fact that for any $p\in (q(\vz),\fz)$ if $q(\vz)\in (1,\fz)$ or if $q(\vz)=1$
and $\vz\notin\aa_1(\rn)$, or for any
$p\in [1,\fz)$ if $q(\vz)=1$ and $\vz\in\aa_1(\rn)$, there exists a positive
constant $C_{(p,\vz)}$ such that for all $f\in L^p_{\vz(\cdot,t)}(\rn)$ and $t\in(0,\fz)$,
\begin{eqnarray*}
\vz(\{x\in\rn:\ |Mf(x)|>\az\},t)\le C_{(p,\vz)}\az^{-p}\int_\rn|f(x)|^{p}\vz(x,t)\,dx,
\end{eqnarray*}
we immediately obtain the following
boundedness of $M$ on $\lv$. We omit the details.

\begin{cor}\label{HL-v}
Let $\vz$ be a
Musielak-Orlicz function with uniformly lower type $p_\vz^-$ and
uniformly upper type $p_\vz^+$ satisfying
$q(\vz)<p_\vz^-\le p_\vz^+<\fz$, where $q(\vz)$ is as in \eqref{2.3}.
Then the Hardy-Littlewood Maximal function $M$ is bounded on $\lv$ and,
moreover, there exists a positive constants $C$
such that for all $f\in L^\vz(\rn)$,
$$\int_\rn\vz(x,Mf(x))\,dx\le C\int_\rn\vz(x,|f(x)|)\,dx.$$
\end{cor}

The \emph{space $L^\vz(\ell^r,\rn)$} is defined to be the set of all $\{f_j\}_{j\in\zz}$
satisfying $[\sum_j|f_j|^r]^{1/r}\in L^\vz(\rn)$ and let
$$\|\{f_j\}_j\|_{L^\vz(\ell^r,\rn)}:=\lf\|\lf[\sum_j|f_j|^r\r]^{1/r}\r\|_\lv.$$
We have the following
vector-valued interpolation theorem of Musielak-Orlicz type.

\begin{thm}\label{inter-vector}
Let $p_1,\,p_2$ and $\vz$ be as in Theorem \ref{inter} and $r\in[1,\fz]$.
Assume that $T$ is a sublinear operator defined on
$L^{p_1}_{\vz(\cdot,1)}(\rn)+L^{p_2}_{\vz(\cdot,1)}(\rn)$
satisfying that for $i\in\{1,2\}$ and all
$\{f_j\}_j\in L^{p_i}_{\vz(\cdot,1)}(\ell^r,\rn)$, $\az\in(0,\fz)$ and $t\in(0,\fz)$,
\begin{eqnarray}\label{2.x2}
&&\vz\lf(\lf\{x\in\rn:\ \lf[\sum_j|Tf_j(x)|^r\r]^{\frac 1r}>\az\r\},t\r)\\
&&\hs\le C_i \az^{-p_i}\int_\rn\lf[\sum_j|f_j(x)|^r\r]^{\frac {p_i}r}\vz(x,t)\,dx,\noz
\end{eqnarray}
where $C_i$ is a positive constant independent of $\{f_j\}_j$, $t$ and $\az$.
Then there exists a positive constant $C$ such that for all $\{f_j\}_j\in L^\vz(\ell^r,\rn)$,
$$\int_\rn\vz\lf(x,\lf[\sum_j|Tf_j(x)|^r\r]^{1/r}\r)\,dx
\le C \int_\rn\vz\lf(x,\lf[\sum_j|f_j(x)|^r\r]^{1/r}\r)\,dx.$$
\end{thm}

\begin{proof}
For all $\{f_j\}_j\in L^\vz(\ell^r,\rn)$ and $x\in\rn$, let
$$n_j(x):=\frac{f_j(x)}{[\sum_j|f_j(x)|^r]^{1/r}}
\quad {\rm if}\quad \lf[\sum_j|f_j(x)|^r\r]^{1/r}\neq0,$$
and $n_j(x)=0$ otherwise.
Then $[\sum_j|n_j(x)|^r]^{1/r}=1$ for all $x\in\rn$.
Consider the operator
$$A(g):=\lf[\sum_j|T(gn_j)|^r\r]^{1/r},$$
where $g\in L_{\vz(\cdot,1)}^{p_1}(\rn)+L_{\vz(\cdot,1)}^{p_2}(\rn)$.
Then, for all
$g_1,\,g_2\in L_{\vz(\cdot,1)}^{p_1}(\rn)+L_{\vz(\cdot,1)}^{p_2}(\rn)$
and $x\in\rn$, by the sublinear property of $T$
and Minkowski's inequality, we see that
\begin{eqnarray*}
A(g_1+g_2)(x)&&=\lf[\sum_j|T((g_1+g_2)n_j)(x)|^r\r]^{1/r}\\
&&\le \lf\{\sum_j\lf[|T(g_1n_j)(x)|+|T(g_2n_j)(x)|\r]^r\r\}^{1/r}\\
&&\le\lf[\sum_j|T(g_1n_j)(x)|^r\r]^{1/r}+\lf[\sum_j|T(g_2n_j)(x)|^r\r]^{1/r}\\
&&=A(g_1)(x)+A(g_2)(x).
\end{eqnarray*}
Thus, $A$ is sublinear. Moreover, by \eqref{2.x2},
we further conclude that for all $i\in\{1,2\}$, $\az\in(0,\fz)$,
$t\in(0,\fz)$ and
$g\in L_{\vz(\cdot,1)}^{p_1}(\rn)+L_{\vz(\cdot,1)}^{p_2}(\rn)$,
\begin{eqnarray*}
\vz(\{x\in\rn:\ |A(g)(x)|>\az\},t)&&=
\vz\lf(\lf\{x\in\rn:\ \lf[\sum_j|T(gn_j)(x)|^r\r]^{1/r}>\az\r\},t\r)\\
&&\ls \az^{-p_i}\int_\rn\lf[\sum_j|gn_j(x)|^r\r]^{p_i/r}\vz(x,t)\,dx\\
&&\ls \az^{-p_i}\int_\rn|g(x)|^{p_i}\vz(x,t)\,dx,
\end{eqnarray*}
which implies that $A$ satisfies \eqref{2.x1}.
Thus, if setting $g=[\sum_j|f_j|^r]^{1/r}$, from Theorem \ref{inter},
we deduce that
\begin{eqnarray*}
\int_\rn\vz\lf(x,\lf[\sum_j|Tf_j(x)|^r\r]^{1/r}\r)\,dx
&&=\int_\rn\vz(x,|Ag(x)|)\,dx
\ls\int_\rn\vz(x,|g(x)|)\,dx\\
&&\ls \int_\rn\vz\lf(x,\lf[\sum_j|f_j(x)|^r\r]^{1/r}\r)\,dx,
\end{eqnarray*}
which completes the proof of Theorem \ref{inter-vector}.
\end{proof}

By using Theorem \ref{inter-vector} and \cite[Theorem 3.1(a)]{aj80},
we immediately obtain the following Musielak-Orlicz
Fefferman-Stein vector-valued inequality,
which, when $\vz(x,t):=t^p$ for all $t\in(0,\fz)$ and $x\in\rn$ with $p\in(1,\fz)$,
was obtained by Fefferman and Stein in \cite[Theorem 1]{fs71}
and, when $\vz(x,t):=w(x)t^p$ for all $t\in(0,\fz)$ and $x\in\rn$
with $p\in(1,\fz)$, $q\in(1,\fz)$ and $w\in A_q(\rn)$,
by Andersen and John in \cite[Theorem 3.1]{aj80}.
We point out that to apply Theorem \ref{inter-vector},
we need $r\in(1,\fz]$.

\begin{thm}\label{f-s}
Let $r\in(1,\fz]$,
$\vz$ be a Musielak-Orlicz function with uniformly lower type $p_\vz^-$
and upper type $p_\vz^+$, $q\in(1,\infty)$ and $\varphi\in\mathbb A_q(\rn)$.
If $q(\vz)<p_\vz^-\le p_\vz^+<\fz$,
then there exists a
positive constant $C$ such that,
for all $\{f_j\}_{j\in\zz}\in L^\vz(\ell^r,\rn)$,
$$\int_\rn\vz\lf(x,\lf\{\sum_{j\in\zz}\lf[M(f_j)(x)\r]^{r}\r\}^{1/r}\r)\,dx
\le C\int_\rn\vz\lf(x,\lf[\sum_{j\in\zz}|f_j(x)|^r\r]^{1/r}\r)\,dx.$$
\end{thm}

\subsection{Musielak-Orlicz Hardy spaces\label{s2.3}}

\hskip\parindent In what follows, we denote by $\cs(\rn)$ the
\emph{space of all Schwartz functions} and
by $\cs'(\rn)$ its \emph{dual space} (namely, the \emph{space of all tempered distributions}).
For $m\in\nn$, define
$$\cs_m(\rn):=\lf\{\psi\in\cs(\rn):\ \sup_{x\in\rn}\sup_{
\bz\in\zz^n_+,\,|\bz|\le m+1}(1+|x|)^{(m+2)(n+1)}|\partial^\bz_x\psi(x)|\le1\r\}.
$$
Then for all $f\in\cs'(\rn)$, the \emph{nontangential grand maximal function}, $f^\ast_m$,
of $f$ is defined by setting, for all $x\in\rn$,
\begin{equation}\label{2.x3}
f^\ast_m(x):=\sup_{\psi\in\cs_m(\rn)}\sup_{|y-x|<t,\,t\in(0,\fz)}|f\ast\psi_t(y)|,
\end{equation}
where for all $t\in(0,\fz)$, $\psi_t(\cdot):=t^{-n}\psi(\frac{\cdot}{t})$.
When
\begin{equation}\label{2.y1}
m(\vz):=\lfz n[q(\vz)/i(\vz)-1]\rfz,
\end{equation}
where $q(\vz)$ and $i(\vz)$ are,
respectively, as in \eqref{2.3} and \eqref{2.1}, we \emph{denote $f^\ast_{m(\vz)}$ simply
by $f^\ast$.}

Now we recall the definition of the Musielak-Orlicz Hardy space $H^\vz(\rn)$ introduced
by Ky \cite{ky} as follows.

\begin{defn}\label{def2}
Let $\vz$ be a growth function. The \emph{Musielak-Orlicz Hardy space
$H^\vz(\rn)$} is defined to be the space of all $f\in\cs'(\rn)$ such that
$f^\ast\in L^\vz(\rn)$
with the \emph{quasi-norm}
$$\|f\|_{H^\vz(\rn)}:=\|f^\ast\|_{L^\vz(\rn)}.$$
\end{defn}

When $\varphi(x,t)=w(x)\Phi(t)$ for all $x\in\rn$ and
$t\in(0,\fz)$, with $w$ being a Muckenhoupt
weight and $\Phi$ an Orlicz function, $\hv$ is just the
weighted Hardy-Orlicz space which includes the classical
Hardy-Orlicz spaces of Janson \cite{j80} ($w=1$ in this
context) and the classical weighted Hardy spaces of
Garc\'{i}a-Cuerva \cite{gr85}, Str\"{o}mberg and
Torchinsky \cite{st89} ($\Phi(t):=t^p$ for all $t\in(0,\fz)$ in this context).

In order to introduce the atomic Musielak-Orlicz Hardy space,
Ky \cite{ky} introduced the following local Musielak-Orlicz space.

\begin{defn}\label{den3}
For any cube $Q$ in $\mathbb R^n$, the \emph{space $L_\varphi^q(Q)$} for
$q\in[1,\infty]$ is defined to be the set of all measurable functions $f$ on
$\mathbb R^n$ supported in $Q$ such that
\begin{eqnarray*}
\; \|f\|_{L_\varphi^q(Q)}:= \left\lbrace
\begin{array}{l l}
\displaystyle \sup_{t\in(0,\fz)}\lf[\frac1{\varphi(Q,t)}{\int_{\mathbb
R^n}|f(x)|^q\varphi(x,t)dx}\r]^{1/q}<\infty,
\ &q\in [1,\fz);\\
\\
\|f\|_{L^\infty(\rn)}<\infty,
&q=\infty.
\end{array} \right.
\end{eqnarray*}
\end{defn}

Now, we recall the atomic Musielak-Orlicz Hardy spaces introduced by Ky \cite{ky}
as follows.

\begin{defn}\label{atom}
A triplet $(\varphi,q,s)$ is called \emph{admissible}, if
$q\in(q(\varphi),\infty]$ and $s\in\mathbb N$ satisfies $s\geq
m(\varphi)$. A measurable function $a$ is called a \emph{$(\varphi,q,s)$-atom}
if it satisfies the following three conditions:

{\rm (i)} $a\in L_\varphi^q(Q)$ for some cube $Q$;

{\rm (ii)} $\|a\|_{L_\varphi^q(Q)}\le \|\chi_Q\|_{L^\varphi(\rn)}^{-1}$;

{\rm (iii)} $\int_{\mathbb R^n}a(x)x^\alpha dx=0$ for any $|\alpha|\le s$.

The \emph{atomic Musielak-Orlicz Hardy space}
$H_{\mathrm{at}}^{\varphi,q,s}(\mathbb R^n)$ is defined to be the space of all
$f\in\mathcal{S}'(\mathbb R^n)$ that can be
represented as a sum of multiples of $(\varphi,q,s)$-atoms, that
is, $f=\sum_j b_j$ in $\mathcal{S}'(\mathbb R^n)$,
where, for each $j$, $b_j$ is a multiple of some $(\varphi,q,s)$-atom supported in
some cube $Q_j$, with the property
$\sum_j\varphi(Q_j,\|b_j\|_{L_\varphi^q(Q_j)})<\infty.$
For any given sequence of multiples of $(\varphi,q,s)-$atoms,
$\{b_j\}_j$, let
$$\Lambda_q(\{b_j\}_j):=\inf\lf\{\lambda>0:\ \sum_j
\varphi\lf(Q_j,\frac{\|b_j\|_{L_\varphi^q(Q_j)}}{\lambda}\r)\le
1\r\}$$
and then define
$$ \|f\|_{H_{\mathrm{at}}^{\varphi,q,s}(\rn)}:=\inf\lf\{\Lambda_q(\{b_j\}_j):\ f=\sum_jb_j\quad\text{in
}\,\ \mathcal{S}'(\mathbb R^n)\r\},$$
where the infimum is taken over all decompositions of $f$ as above.
\end{defn}

The following Proposition \ref{prop0} is just \cite[Theorem 3.1]{ky}.

\begin{prop}\label{prop0}
Let $(\varphi,q,s)$ be admissible. Then $H^\varphi(\mathbb
R^n)=H_{\mathrm{at}}^{\varphi,q,s}(\mathbb R^n)$ with equivalent norms.
\end{prop}

\section{Maximal function characterizations of $H^\varphi(\mathbb R^n)$\label{s3}}

\hskip\parindent In this section, we establish some maximal function
characterizations of $\hv$.
First, we recall the notions of the vertical, the tangential and the nontangential
maximal functions. In what follows, let the \emph{space} $\cd(\rn)$ be the space of all $C^\fz(\rn)$
functions with compact support, endowed with the inductive limit
topology, and $\cd'(\rn)$ its \emph{topological dual space},
endowed with the weak-$\ast$ topology.

\begin{defn}\label{d3.6}
Let
\begin{eqnarray}\label{3.2}
\psi_0\in\cd(\rn)\,\, \text{and}\,\,\int_{\rn}\psi_0 (x)\,dx\neq0.
\end{eqnarray}
For $j\in\zz$, $A,\,B\in[0,\fz)$ and $y\in\rn$, let
$m_{j,\,A,\,B}(y):=(1+2^j |y|)^A 2^{B|y|}$. The
{\it vertical maximal function} $\psi_0^{+}(f)$ of $f$ associated to $\psi_0$ is
defined by setting, for all $x\in\rn$,
\begin{eqnarray}\label{3.3}
\psi_0^{+}(f)(x):= \sup_{j\in\zz}|(\psi_0)_j \ast f(x)|,
\end{eqnarray}
the {\it tangential Peetre-type  maximal function
$\psi^{\ast\ast}_{0,\,A,\,B}(f)$} of $f$ associated to $\psi_0$ is
defined by setting, for all $x\in\rn$,
\begin{eqnarray}\label{3.4}
\psi^{\ast\ast}_{0,\,A,\,B}(f)(x):=\sup_{j\in\zz,\,y\in\rn}
\frac{|(\psi_0)_j \ast f(x-y)|}{m_{j,\,A,\,B}(y)}
\end{eqnarray}
and the {\it nontangential maximal function
$(\psi_0)^{\ast}_{\triangledown}(f)$} of $f$ associated to $\psi_0$ is
defined by setting, for all $x\in\rn$,
$$(\psi_0)^{\ast}_{\triangledown}(f)(x):= \sup_{|x-y|<t}|(\psi_0)_t
\ast f(y)|,$$
here and in what follows, for all $x\in\rn$,
$(\psi_0)_j(x):=2^{jn}\psi_0 (2^j x)$ for all $j\in\zz$ and $(\psi_0)_t
(x):=\frac{1}{t^n}\psi_0 (\frac{x}{t})$ for all $t\in(0,\fz)$.
\end{defn}

Obviously, for all $x\in\rn$, we have
$\psi_0^{+}(f)(x)\le(\psi_0)^{\ast}_{\triangledown}(f)(x)
\ls\psi^{\ast\ast}_{0,\,A,\,B}(f)(x)$.

In order to establish the vertical or the
nontangential maximal function characterizations of
$H^{\vz}(\rn)$, we first establish some inequalities in the norm
of $L^{\vz}(\rn)$ involving the maximal functions
$\psi^{\ast\ast}_{0,\,A,\,B}(f),\,\psi_0^{+}(f)$ and
$f^*$. We begin with some technical lemmas and the
following Lemma \ref{l3.1} is just \cite[Theorem\,1.6]{r01}.

\begin{lem}\label{l3.1}
Let $\psi_0$ be as in \eqref{3.2} and $\psi(x):=\psi_0
(x)-\frac{1}{2^n}\psi_0 (\frac{x}{2})$ for all $x\in\rn$. Then for
any given integer $L\in\nn$, there exist $\eta_0,\,\eta\in\cd(\rn)$
such that $L_{\eta}\ge L$ and, for all $f\in\cd'(\rn)$,
$$f=\eta_0 \ast\psi_0 \ast f+\sum_{j\in\nn}\eta_j\ast\psi_j\ast f\quad
{\rm in}\quad \cd'(\rn).$$
\end{lem}

For $f\in L^1_{\loc}(\rn)$, $B\in[0,\fz)$ and $x\in\rn$, let
$$K_B f(x):=\int_{\rn}|f(y)|2^{-B|x-y|}\,dy,$$
here and in what follows,
$L^1_{\loc}(\rn)$ denotes the \emph{space of all locally integrable
functions on $\rn$}.

\begin{lem}\label{l3.2}
Let $p\in(1,\fz)$, $q\in(1,\fz]$ and $\vz\in \mathbb{A}_p(\rn)$.
Then there exist positive
constants $C$ and $B_0:=B_0 (\vz,n)$ such that for all $t\in(0,\fz)$, $B\ge
B_0/p$ and $\{f^j\}_j\in L^p_{\vz(\cdot,t)}(\ell^q,\rn)$,
$$\lf\|\{K_B (f^j)\}_j\r\|_{L^p_{\vz(\cdot,t)}(\ell^q,\rn)}\le C
\lf\|\{f^j\}_j\r\|_{L^p_{\vz(\cdot,t)}(\ell^q,\rn)}.$$
\end{lem}

Lemma \ref{l3.2} is just \cite[Lemma\,2.11]{r01}.

\begin{lem}\label{l3.x1}
Let $\psi_0$ be as in \eqref{3.2} and $r\in(0,\fz)$. Then
for any $A,B\in[0,\fz)$, there exists a positive constant $C$, depending only
on $n,\,r,\,\psi_0,\,A$ and $B$, such that for all
$f\in\cs'(\rn)$, $j\in\zz$ and $x\in\rn$,
$$|(\psi_0)_j*f(x)|^r\le C\sum_{k=j}^{\fz}
2^{(j-k)Ar}2^{kn}\int_\rn\frac{|(\psi_0)_k \ast
f(x-y)|^r}{m_{j,\,Ar,\,Br}(y)}\,dy.$$
\end{lem}

\begin{proof}
When $j\in\nn$, Lemma \ref{l3.x1} is just \cite[Lemma 2.9]{r01}. We now show
the conclusion of Lemma \ref{l3.x1} for all $j\in\zz$.

By Lemma \ref{l3.1}, there exist $\eta_0, \eta\in\cd(\rn)$ such that
$L_\eta\ge A$ and, for all $f\in\cs'(\rn)$,
$$f=\eta_0*\psi_0*f+\sum_{k\in\nn}\eta_k*\psi_k*f.$$
We dilate this identity with $2^j$, $j\in\zz$, namely, for $\phi\in\cs(\rn)$,
$\la f_j,\phi\ra=\la f, 2^{-jn}\phi(2^{-j}\cdot)\ra$. By an elementary
calculation, we see that, for all $j\in\zz$ and $f\in\cs'(\rn)$,
$$f_j=(\eta_0)_j*(\psi_0)_j*f_j+\sum_{k\in\nn}\eta_{k+j}*\psi_{k+j}*f_j.$$
We rewrite the above equality to conclude that, for all $j\in\zz$ and $f\in\cs'(\rn)$,
\begin{eqnarray}\label{3.x1}
f=(\eta_0)_j*(\psi_0)_j*f+\sum_{k\in\nn}\eta_{k+j}*\psi_{k+j}*f.
\end{eqnarray}
Then, replacing \cite[(2.12)]{r01} by \eqref{3.x1},
similar to the proof of \cite[Lemma 2.9]{r01},
we obtain the conclusion of Lemma \ref{l3.x1},
which completes the proof of Lemma \ref{l3.x1}.
\end{proof}

The proof of the following lemma is quite similar to that of \cite[Lemma
2.10]{r01}, and we omit the details.

\begin{lem}\label{l3.3}
Let $\psi_0$ be as in \eqref{3.2} and $r\in(0,\fz)$. Then there
exists a positive constant $A_0$, depending only on the support of
$\psi_0$, such that for any $A\in(\max\{A_0,\frac{n}{r}\},\fz)$ and
$B\in[0,\fz)$, there exists a positive constant $C$, depending only
on $n,\,r,\,\psi_0,\,A$ and $B$, such that for all
$f\in\cs'(\rn)$, $j\in\zz$ and $x\in\rn$,
$$\lf[(\psi_0 )^{\ast}_{j,\,A,\,B}(f)(x)\r]^r\le C\sum_{k=j}^{\fz}
2^{(j-k)(Ar-n)}\lf\{M (|(\psi_0 )_k \ast
f|^r)(x)+K_{Br}(|(\psi_0 )_k \ast f|^r)(x)\r\},$$
where
$$(\psi_0
)^{\ast}_{j,\,A,\,B}(f)(x):=\sup_{y\in\rn}\frac{|(\psi_0 )_j \ast
f(x-y)|}{m_{j,\,A,\,B}(y)} \quad {\rm for\ all}\quad x\in\rn.$$
\end{lem}

\begin{thm}\label{t3.1}
Let $\vz$ be a growth function as in Definition \ref{d2.2},
$R\in(0,\fz)$, $\psi_0$ as
in \eqref{3.2}, $\psi_0^{+}(f),\,
\psi^{\ast\ast}_{0,\,A,\,B}(f)$, and $f^*$ be
respectively as in \eqref{3.3}, \eqref{3.4} and \eqref{2.x3} with $m=m(\vz)$. Let
$A_1 :=\max\{A_0,\,nq(\vz)/i(\vz)\}$, $B_1:= B_0/i(\vz)$ and
integer $N_0:=\lfz2A_1\rfz+1$, where $A_0$ and $B_0$ are
respectively as in Lemmas \ref{l3.3} and \ref{l3.2}. Then for any
$A\in(A_1,\fz)$, $B\in(B_1,\fz)$ and integer $N\ge N_0$, there
exists a positive constant $C$, depending only on
$A,\,B,\,N,\,R,\,\psi_0,\,\vz$ and $n$, such that for all
$f\in\cs'(\rn)$,
\begin{eqnarray}\label{3.15}
\lf\|\psi^{\ast\ast}_{0,\,A,\,B}(f)\r\|_{L^{\vz}(\rn)}\le C
\lf\|\psi_0^{+}(f)\r\|_{L^{\vz}(\rn)}
\end{eqnarray}
and
\begin{eqnarray}\label{3.16}
\lf\| f^*\r\|_{L^{\vz}(\rn)}\le C
\lf\|\psi_0^{+}(f)\r\|_{L^{\vz}(\rn)}.
\end{eqnarray}
\end{thm}

\begin{proof}
Let $f\in\cs'(\rn)$. First, we prove \eqref{3.15}. Let
$A\in(A_1,\fz)$ and $B\in(B_1,\fz)$. By $A_1
=\max\{A_0,\,nq(\vz)/i(\vz)\}$ and $B_1=B_0 /i(\vz)$, we know
that there exists $r_0\in(0,\frac{i(\vz)}{q(\vz)})$ such that
$A>\frac{n}{r_0}$ and $Br_0>\frac{B_0}{q(\vz)}$,  where $A_0$ and
$B_0$ are respectively as in Lemmas \ref{3.2} and \ref{l3.2}. Thus,
by Lemma \ref{l3.3}, for all $x\in\rn$, we know that
\begin{eqnarray}\label{3.17}
\lf[(\psi_0)^{\ast}_{j,\,A,\,B}(f)(x)\r]^{r_0}\ls&&\sum_{k=j}^{\fz}
2^{(j-k)(Ar_0 -n)}\\
&&\hs\times\lf\{M\lf(|(\psi_0)_k\ast f|^{r_0}\r)(x)+K_{Br_0}
\lf(|(\psi_0)_k\ast f|^{r_0}\r)(x)\r\},\noz
\end{eqnarray}
where $M$ is the Hardy-Littlewood maximal function.
Let $\psi^{+}_0 (f)$ and $\psi^{\ast\ast}_{0,\,A,\,B}(f)$ be
respectively as in \eqref{3.3} and \eqref{3.4}. We notice that for
any $x\in\rn$ and $k\in\nn$, $|(\psi_0)_k \ast f (x)|\le\psi^{+}_0
(f)(x)$, which, together with \eqref{3.17}, implies that for all
$x\in\rn$,
\begin{eqnarray}\label{3.18}
\lf[\psi^{\ast\ast}_{0,\,A,\,B}(f)(x)\r]^{r_0}\ls
M \lf([\psi^{+}_0 (f)]^{r_0})(x)+K_{Br_0}([\psi^{+}_0
(f)]^{r_0}\r)(x).
\end{eqnarray}
By $r_0 <\frac{i(\vz)}{q(\vz)}$,
we see that there exist $q\in(q(\vz),\fz)$ and $p_0\in(0,i(\vz))$
such that $r_0 q<p_0$, $\vz\in \aa _q (\rn)$ and $\vz$ is of
uniformly lower type $p_0$. Thus, $\wz\vz(x,t):=\vz(x,t^{1/r_0})$
is of uniformly lower type $p_0/r_0$. Then from
\eqref{3.18}, Lemmas \ref{lem1}(i) and \ref{l3.2}, Theorem \ref{inter-vector}
and Corollary \ref{HL-v}, together with the fact that $p_0/r_0>q\ge1$,
we deduce that
\begin{eqnarray*}
&&\int_{\rn}\vz\lf(x,\psi^{\ast\ast}_{0,\,A,\,B}(f)(x)\r)\,dx\\
\nonumber &&\hs\ls\int_{\rn}\vz\lf(x,\lf\{M \lf([\psi^{+}_0
(f) ]^{r_0}\r)(x)\r\}^{1/{r_0}}\r)\,dx\\
&&\hs\hs+\int_{\rn}\vz\lf(x,\lf\{K_{Br_0}\lf([\psi^{+}_0
(f)]^{r_0}\r)(x)\r\}^{1/{r_0}}\r)\,dx
\ls\int_{\rn}\vz\lf(x,\psi^+_0 (f)(x)\r)\,dx.\noz
\end{eqnarray*}
Replacing $f$ by
$f/\lz$ with $\lz\in(0,\fz)$ in the above inequality and noticing
that $\psi^{\ast\ast}_{0,\,A,\,B}(f/\lz\!)
\!\!=\psi^{\ast\ast}_{0,\,A,\,B}(f)/\lz$ and $\psi^+_0 (f/\lz)=\psi^+_0
(f)/\lz$, we see that
\begin{eqnarray}\label{3.19}
\int_{\rn}\vz\lf(x,\frac{\psi^{\ast\ast}_{0,\,A,\,B}(f)(x)}{\lz}\r)\,dx\ls
\int_{\rn}\vz\lf(x,\frac{\psi^+_0 (f)(x)}{\lz}\r)\,dx,
\end{eqnarray}
which, together with the arbitrariness of $\lz\in(0,\fz)$, implies
\eqref{3.15}.

Now, we prove \eqref{3.16}. By $N_0 =\lfz2A_1\rf+1$, we know
that there exists $A\in(A_1,\fz)$ such that $2A<N_0$. In the remainder of
this proof, we fix $A\in(A_1,\fz)$ satisfying $2A<N_0$ and
$B\in(B_1,\fz)$. Let integer $N\ge N_0$. For any
$\gz\in\cs_{N}(\rn)$, $t\in(0,1)$ and $j\in\zz_+$, from Lemma
\ref{l3.1} and \eqref{3.x1}, it follows that
\begin{eqnarray}\label{3.20}
\gz_t*f=\gz_t \ast(\eta_0)_j\ast(\psi_0)_j\ast
f+\sum_{k=j+1}^{\fz}\gz_t \ast\eta_k\ast\psi_k \ast f,
\end{eqnarray}
where $\eta_0,\,\eta\in\cd(\rn)$ with $L_{\eta}\ge N$ and $\psi$ is as
in Lemma \ref{l3.1}.

For any given $t\in(0,1)$ and $x\in\rn$, let $2^{-j_0-1}\le
t<2^{-j_0}$ for some $j_0\in\zz_+$ and $z\in\rn$ satisfy $|z-x|<t$.
Then, by \eqref{3.20}, we conclude that
\begin{eqnarray}\label{3.21}
|\gz_t*f(z)|&\le&\left|\gz_t \ast(\eta_0)_{j_0}\ast(\psi_0)_{j_0}
\ast f(z)\r|+\sum^{\fz}_{k=j_0+1}\left|\gz_t \ast\eta_k\ast\psi_k \ast f(z)\r|\\
 &\le&\int_{\rn}\left|\gz_t \ast(\eta_0)_{j_0}
(y)\r|\left|(\psi_0)_{j_0}\ast f(z-y)\r|\,dy\nonumber\\
&&+\sum_{k=j_0+1}^{\fz}\int_{\rn}\left|\gz_t \ast\eta_k
(y)\r|\left|\psi_k\ast f(z-y)\r|\,dy =:\mathrm{I_1}+\mathrm{I_2}.\noz
\end{eqnarray}

To estimate $\mathrm{I_1}$, from
\begin{eqnarray*}
\psi^{\ast\ast}_{0,\,A,\,B}(f)(x)&=&\sup_{j\in\zz_+,\,y\in\rn}
\frac{|(\psi_0)_j\ast
f(x-y)|}{m_{j,\,A,\,B}(y)}=\sup_{j\in\zz_+,\,y\in\rn}\frac{|(\psi_0)_j\ast
f(z-y)|}{m_{j,\,A,\,B}(y+x-z)},
\end{eqnarray*}
we infer that
$\left|(\psi_0)_{j_0}\ast f(z-y)\r|\le\psi^{\ast\ast}_{0,\,A,\,B}(f)(x)
m_{j_0,\,A,\,B}(y+x-z),$
which, together with the facts that
$m_{j_0,\,A,\,B}(y+x-z)\le m_{j_0,\,A,\,B}(x-z)m_{j_0,\,A,\,B}(y)$
and $m_{j_0,\,A,\,B}(x-z)\ls2^{A}$, implies that
$|(\psi_0)_{j_0}\ast f(z-y)|\ls2^{A}\psi^{\ast\ast}_{0,\,A,\,B}(f)(x)
m_{j_0,\,A,\,B}(y).$
Thus, we have
\begin{eqnarray*}
\mathrm{I_1}\ls2^{A}\left\{\int_{\rn}|\gz_t \ast(\eta_0)_{j_0}
(y)|m_{j_0,\,A,\,B}(y)\,dy\r\}\psi^{\ast\ast}_{0,\,A,\,B}(f)(x).
\end{eqnarray*}

To estimate $\mathrm{I_2}$, by the definition of $\psi$,
we see that, for any $k\in\nn$,
$$\left|\psi_k\ast
f(z-y)\r|\le\left|(\psi_0)_k\ast f(z-y)\r|+\left|(\psi_0)_{k-1}\ast
f(z-y)\r|.$$
By the definition of $\psi^{\ast\ast}_{0,\,A,\,B}(f)$
and the facts that for any $k\in\nn$,
$$m_{k,\,A,\,B}(y+x-z)\le
m_{k,\,A,\,B}(x-z)m_{k,\,A,\,B}(y)$$
and
$m_{k,\,A,\,B}(x-z)\ls2^{(k-j_0)A}$, we conclude that
\begin{eqnarray*}
|(\psi_0)_k\ast f(z-y)|\le
\psi^{\ast\ast}_{0,\,A,\,B}(f)(x)m_{k,\,A,\,B}(y+x-z)
\ls2^{(k-j_0)A} m_{k,\,A,\,B}(y)\psi^{\ast\ast}_{0,\,A,\,B}(f)(x).
\end{eqnarray*}
Similarly, we also have
$|(\psi_0)_{k-1}\ast f(z-y)|\ls2^{(k-j_0)A}
m_{k,\,A,\,B}(y)\psi^{\ast\ast}_{0,\,A,\,B}(f)(x).$
Thus,
\begin{eqnarray*}
\mathrm{I_2}\ls\sum_{k=j_0+1}^{\fz}2^{(k-j_0)A}\left\{\int_{\rn}|\gz_t
\ast\eta_k
(y)|m_{k,\,A,\,B}(y)\,dy\r\}\,\psi^{\ast\ast}_{0,\,A,\,B}(f)(x).
\end{eqnarray*}

From \eqref{3.21} and the above estimates of $\mathrm{I_1}$ and
$\mathrm{I_2}$, it follows that
\begin{eqnarray}\label{3.22}
|\gz_t*f(z)|&\ls&\left\{\int_{\rn}|\gz_t \ast(\eta_0)_{j_0}
(y)|m_{j_0,\,A,\,B}(y)\,dy\r.\\
&&+\left.\sum_{k=j_0+1}^{\fz}2^{(k-j_0)A}\int_{\rn}|\gz_t \ast\eta_k
(y)|m_{k,\,A,\,B}(y)\,dy\r\}\psi^{\ast\ast}_{0,\,A,\,B}(f)(x).\nonumber
\end{eqnarray}
Assume that $\supp(\eta_0)\subset B(0,R_0)$. Then
$\supp((\eta_0)_j)\subset B(0,2^{-j}R_0)$ for all $j\in\zz_+$.
Moreover, by $\supp(\gz)\subset B(0,R)$ and $2^{-j_0-1}\le
t<2^{-j_0}$, we see that
$\supp(\gz_t)\subset B(0,2^{-j_0}R).$
From this, we further deduce that
$\supp(\gz_t\ast(\eta_0)_{j_0})\subset B(0,2^{-j_0}(R_0
+R))$ and
\begin{eqnarray*}
|\gz_t\ast(\eta_0)_{j_0} (y)|\ls\int_{\rn}|\gz_t (s)||(\psi_0)_{j_0}
(y-s)|\,ds\ls2^{j_0 n}\int_{\rn}|\gz_t (s)|\,ds\sim2^{j_0 n},
\end{eqnarray*}
which implies that
\begin{eqnarray}\label{3.23}
\qquad\int_{\rn}|\gz_t\ast(\eta_0)_{j_0} (y)|m_{j_0,\,A,\,B}(y)\,dy
\ls2^{j_0 n}\int_{B(0,2^{-j_0}(R_0 +R))}(1+2^{j_0}
|y|)^A 2^{B|y|}\,dy\ls1.
\end{eqnarray}
Moreover, since $\eta$ has vanishing moments up to order $N$, it was
proved in \cite[(2.13)]{r01} that
$\|\gz_t\ast\eta_k\|_{L^{\fz}(\rn)}\ls2^{(j_0-k)N}2^{j_0 n}$
for all $k\in\nn$ with $k\ge j_0 +1$, which, together with the facts that
$N>2A$ and
$\supp(\gz_t\ast\eta_k) \subset B(0,2^{-j_0}R_0
+2^{-k}R),$
implies that
\begin{eqnarray}\label{3.24}
&&\sum_{k=j_0+1}^{\fz}2^{(k-j_0)A}\int_{\rn}|\gz_t \ast\eta_k
(y)|m_{k,\,A,\,B}(y)\,dy\\
 &&\hs\ls\sum_{k=j_0+1}^{\fz}2^{(k-j_0)A}2^{(j_0-k)N}2^{j_0
n} (2^{-j_0}R_0 +2^{-k}R)^n\nonumber\\  &&\hs\hspace{1
em}\times\left[1+2^k
(2^{-j_0}R_0+2^{-k}R)\r]^A 2^{(2^{-j_0}R_0+2^{-k}R)B}
\ls\sum_{k=j_0 +1}^{\fz}2^{(j_0-k)(N-2A)}\ls1.\nonumber
\end{eqnarray}
Thus, from \eqref{3.22}, \eqref{3.23} and \eqref{3.24}, we deduce
that $|\gz_t*f(z)|\ls\psi^{\ast\ast}_{0,\,A,\,B}(f)(x)$. Then,
by the arbitrariness of $t\in(0,1)$ and $z\in B(x,t)$, we know that
$f^*(x)\ls\psi^{\ast\ast}_{0,\,A,\,B}(f)(x),$
which, together with \eqref{3.19}, implies that, for any $\lz\in(0,\fz)$,
$$\int_{\rn}\vz\left(x,f^*(x)/\lz\r)\,dx\ls
\int_{\rn}\vz\left(x,\psi^{+}_0 (f)(x)/\lz\r)\,dx.$$ From this, we
infer that \eqref{3.16} holds, which completes the proof of Theorem
\ref{t3.1}.
\end{proof}

From Theorem \ref{t3.1}, we immediately deduce the following
vertical and the nontangential maximal function
characterizations of $H^{\vz}(\rn)$.
We omit the details.

\begin{thm}\label{max-ch}
Let $\vz$ be a growth function as in Definition \ref{d2.2}, and
$\psi_0$, $\psi_0^+$ and $(\psi_0)^{\ast}_{\triangledown}$
as in Definition \ref{d3.6}. Then the followings are equivalent:
\begin{enumerate}
\vspace{-0.25cm}
\item[\rm(i)] $f\in H^{\vz}(\rn);$
\vspace{-0.25cm}
\item[\rm(ii)] $f\in\cs'(\rn)$ and $\psi^{+}_0 (f)\in L^{\vz}(\rn);$
\vspace{-0.25cm}
\item[\rm(iii)] $f\in\cs'(\rn)$ and $(\psi_0)^{\ast}_{\triangledown}
(f)\in L^{\vz}(\rn).$ \vspace{-0.25cm}
\end{enumerate}
Moreover, for all $f\in H^{\vz}(\rn)$,
$\|f\|_{H^{\vz}(\rn)}\sim\|\psi^{+}_0
(f)\|_{L^{\vz}(\rn)}\sim
\lf\|(\psi_0)^{\ast}_{\triangledown}(f)\r\|_{L^{\vz}(\rn)},$
where the implicit constants are independent of $f$.
\end{thm}

\section{The Littlewood-Paley $g$-function and
$g_\lambda^\ast$-function characterizations of $H^\varphi\!(\mathbb R^n)$\label{s4}}

\hskip\parindent In this section, we establish
the Littlewood-Paley $g$-function and $g_\lz^*$-function
characterizations of $\hv$.

Let $\phi\in\mathcal{S}(\mathbb R^n)$ be a radial function,
$\supp\phi\subset\{x\in\mathbb R^n:\ |x|\le 1\},$
\begin{equation}\label{4.y2}
  \int_{\mathbb R^n}\phi(x)x^\gz \,dx=0
\end{equation}
for all $|\gz|\le m(\vz)$, where $m(\vz)$ is as in \eqref{2.y1}
and, for all $\xi\in\rn\bh\{0\}$,
$$\int_0^\infty|\hat{\phi}(\xi t)|^2\frac{dt}{t}=1.$$
Recall that for all $f\in\mathcal{S}'(\mathbb R^n)$, the \emph{$g$-function},
the \emph{Lusin area integral} and the \emph{$g_\lz^*$-function}, with $\lz\in (1,\fz)$,
of $f$ are
defined, respectively, by setting, for all $x\in\rn$,
$$g(f)(x):=\lf[\int_0^\infty\lf|f\ast\phi_t(y)\r|^2\dt\r]^{1/2},$$
$$S(f)(x):=\lf[\int_0^\infty\int_{\{y\in\rn:\ |y-x|<t\}}
 |f\ast\phi_t(y)|^2\,\frac{dy\,dt}{t^{n+1}}\r]^{1/2}$$
and
$$g_\lz^*(f)(x):=\lf[\int_0^\infty\int_{\mathbb  R^n}\lf(\frac{t}{t+|x-y|}\r)^{\lambda
 n}\lf|f\ast\phi_t(y)\r|^2\,\frac{dy\,dt}{t^{n+1}}\r]^{1/2}.$$

Recall that $f\in\mathcal{S}'(\mathbb R^n)$ is called to \emph{vanish weakly at
infinity}, if for every $\phi\in\mathcal{S}(\mathbb R^n)$,
$f\ast\phi_t\to 0$ in $\mathcal{S}'(\mathbb R^n)$ as $t\to \infty$;
see, for example, \cite[p.\,50]{fs82}.
We have the following useful property of $\hv$, which is just \cite[Lemma 4.12]{hyy11}.

\begin{prop}\label{van}
Let $\vz$ be a growth function.
If $f\in H^\varphi(\mathbb R^n)$, then
$f$ vanishes weakly at infinity.
\end{prop}

The following Proposition \ref{s-ch} is just \cite[Theorem 4.11]{hyy11}.

\begin{prop}\label{s-ch}
Let $\vz$ be a growth function.
Then $f\in\hv$ if and only if $f\in\cs'(\rn)$, $f$ vanishes weakly at infinity
and $S(f)\in\lv$ and, moreover,
$$\frac1C\|S(f)\|_\lv\le\|f\|_\hv\le C\|S(f)\|_\lv$$
with $C$ being a positive constant independent of $f$.
\end{prop}


Similar to the proof of \cite[Lemma 5.1]{ky}, we
easily obtain the following boundedness of the Littlewood-Paley $g$-function
from $\hv$ to $\lv$. We omit the details.

\begin{prop}\label{g-bd}
Let $\vz$ be a growth function.
If $f\in H^\varphi(\mathbb R^n)$, then $g (f)\in
L^\varphi(\rn)$ and, moreover, there exists a positive constant $C$
such that for all $f\in H^\varphi(\rn)$,
$$\|g (f)\|_{L^\varphi(\rn)}\le C \|f\|_{H^\varphi(\rn)}.$$
\end{prop}

We have the following Littlewood-Paley $g$-function characterization of $\hv$.

\begin{thm}\label{g-ch}
Let $\vz$ be as in Definition \ref{d2.2}.
Then $f\in\hv$ if and only if $f\in\cs'(\rn)$, $f$ vanishes weakly at infinity and $g(f)\in\lv$
and, moreover,
$$\frac1C\|g(f)\|_\lv\le\|f\|_\hv\le C\|g(f)\|_\lv$$
with $C$ being a positive constant independent of $f$.
\end{thm}

\begin{proof}
By Propositions \ref{van} through \ref{g-bd},
it suffices to prove that if $f\in\cs'(\rn)$ and
$g(f)\in\lv$, then
$\|S(f)\|_{L^\varphi(\rn)}\ls\|g (f)\|_{L^\varphi(\rn)}.$

For $j\in\zz$, let
$$\cd_j:=\{I\st\rn:\ I {\rm\ is\ a\ dyadic\ cube\ and\ }\ell(I)=2^{-j} \}.$$
From \cite[Theorem 2.1, Lemmas 2.1 and 2.2]{lz11},
we deduce that for any fixed $L\in[0, m(\vz)+1)$,
$K\in\nn$ and $r\in(\frac n{n+K},1]$,
there exists $N\in\nn$ large enough
such that, for all $j\in\zz$,
$f\in\cs'(\rn)$, $u, u^*\in I$
and $x_{\wz I}\in \wz I$,
\begin{eqnarray}\label{4.y1}
|\phi_j*f(u)|&&\!\ls\!\sum_{\wz j\in\zz}\sum_{\wz I\in\cd_{\wz j+N}}
\frac{2^{-|j-\wz j|L}2^{-(\min\{j,\wz j\})K}|\wz I|}{(2^{-\min\{j,\wz j\}}+|u-x_I|)^{n+K}}
|(\phi_{\wz j}*f)(x_{\wz I})|\\
&&\!\ls\!\sum_{\wz j\in\zz}2^{-|j-\wz j|L
+n(1-\frac1r)[\min(\wz j,j)-\wz j]}\noz\\
&&\hs\times\lf[M\lf(\!\lf[\sum_{\wz I\in\cd_{\wz j+N}}
|\phi_{\wz j}*f(x_{\wz I})|\chi_{\wz I}\r]^r\!\r)(u^*)\r]^{1/r}.\noz
\end{eqnarray}
We point out that $L$ in \cite[Lemma 2.1]{lz11} must
be strictly less than $M+1$, which can not be arbitrary,
as claimed in \cite[Lemma 2.1]{lz11}. This is why we need
to restrict $L\in[0, m(\vz)+1)$.

Let $i(\vz)$,
$q(\vz)$ and $m(\vz)$ be, respectively, as in
\eqref{2.1}, \eqref{2.3} and \eqref{2.y1}. By
$m(\vz)+1=\lfz n[q(\vz)/i(\vz)-1]\rfz+1$ and the definitions of $q(\vz)$ and $i(\vz)$,
we know that there exist $q_0\in(q(\vz),\fz)$ and $p_0\in(0,i(\vz))$ such that
$\vz\in\aa_{q_0}(\rn)$, $\vz$ is uniformly lower type $p_0$ and
$L:= n(q_0/p_0-1)<m(\vz)+1$.
Then $\frac n{n+L}=\frac {p_0}{q_0}<p_0$.
Choosing $r\in(\frac n{n+L},p_0)$, by \eqref{4.y1},
we further conclude that, for all $x\in\rn$,
\begin{eqnarray}\label{4.16}
S(f)(x)&&\sim\lf[\sum_{j\in\zz}2^{jn}\int_{B(x,2^{-j})}|\phi_j*f(y)|^2\,dy\r]^{1/2}\\
&&\ls\lf(\sum_{j\in\zz}2^{jn}
\int_{B(x,2^{-j})}
\lf\{\sum_{\wz j\in\zz}2^{-|j-\wz j|L
+n(1-\frac1r)[\max(\wz j,j)-j]}\r.\r.\noz\\
&&\hs\times\lf.\lf.\lf[M\lf(\lf[\sum_{\wz I\in\cd_{\wz j+N}}|\phi_{\wz j}*f(x_{\wz I})|
\chi_{\wz I}\r]^r\r)(x)\r]^{1/r}\r\}^2\,dy\r)^{1/2}\noz\\
&&\sim\lf\{\sum_{\wz j\in\zz}
\lf[M\lf(\lf[\sum_{\wz I\in\cd_{\wz j+N}}|\phi_{\wz j}*f(x_{\wz I})|
\chi_{\wz I}\r]^r\r)(x)\r]^{2/r}\r\}^{1/2}.\noz
\end{eqnarray}

Choose $K$ large enough such that $\frac n{n+K}<p_0$ and $r\in(\max
\{\frac n{n+L},\frac n{n+K}\},p_0)$.
Let $\wz\vz(x,t):=\vz(x,t^{1/r})$ for all $x\in\rn$ and $t\in [0,\fz)$,
and $p_1\in[1,\fz)$. From the fact that $\vz$ is of uniformly upper type $p_1$ and lower
type $p_0$, it follows that $\wz\vz$ is of uniformly
upper type $p_1/r$ and lower type $p_0/r$. Then, by Theorem \ref{f-s},
together with $p_1/r>p_0/r>1$,
we conclude that
\begin{eqnarray*}
\int_\rn\vz\lf(x,\lf\{\sum_{j\in\zz}\lf[M(f_j^r)(x)\r]^{2/r}\r\}^{1/2}\r)\,dx
\ls\int_\rn\vz\lf(x,\lf[\sum_{j\in\zz}|f_j(x)|^2\r]^{1/2}\r)\,dx,
\end{eqnarray*}
which, together with \eqref{4.16}, further implies that
\begin{eqnarray*}
&&\int_\rn\vz(x,S(f)(x))\,dx\\
&&\hs\ls\int_\rn\vz\lf(x,\lf\{\sum_{\wz j\in\zz}\lf[M\lf(\lf[\sum_{\wz I\in\cd_{\wz j+N}}
|\phi_{\wz j}*f(x_{\wz I})|
\chi_{\wz I}\r]^r\r)(x)\r]^{2/r}\r\}^{1/2}\r)\,dx\\
&&\hs\ls\int_\rn\vz\lf(x,\lf\{\sum_{\wz j\in\zz}\sum_{\wz I\in\cd_{\wz j+N}}
\lf[|\phi_{\wz j}*f(x_{\wz I})|\chi_{\wz I}(x)\r]^2\r\}^{1/2}\r)\,dx
\sim\int_\rn\vz\lf(x,g (f)(x)\r)\,dx,
\end{eqnarray*}
where, in the last step, we used the arbitrariness of $x_{\wz I}\in \wz I$.
This finishes the proof of Theorem \ref{g-ch}.
\end{proof}

It is easy to see that $S (f)(x)\le g_\lz^*(f)(x)$ for all $x\in\rn$,
which, together with Proposition \ref{s-ch}, immediately implies the
following conclusion.

\begin{prop}\label{prop3}
Let $\vz$ be as in Definition \ref{d2.2} and $\lz\in (1,\fz)$.
If $f\in\mathcal{S}'(\mathbb R^n)$ vanishes weakly at infinity and
$g_\lz^*(f)\in L^\varphi(\rn)$, then $f\in
H^\varphi(\rn)$ and, moreover,
$$\|f\|_{\hv}\le C\|g_\lz^*(f)\|_{L^\varphi(\rn)}$$
with $C$ being a positive constant independent of $f$.
\end{prop}

Next we consider the boundedness of $g_\lz^*$ on $\lv$. To this end,
we need to introduce the following variant of the Lusin area function
$S$. For all $\az\in(0,\fz)$, $f\in \cs'(\rn)$ and $x\in\rn$, let
$$ S_{\alpha}(f)(x):=\lf[\int_0^\infty\int_{\{y\in\rn:\ |y-x|<\alpha t\}}
\lf|f\ast\phi_t(y)\r|^2(\az t)^{-n}\dyt\r]^{1/2}.$$

The following technical lemma plays a key role to obtain
the $g_\lz^\ast$-function characterization of $\hv$, whose
proof was motivated by Folland and Stein \cite[p.\,218, Theorem (7.1)]{fs82}
and Aguilera and Segovia \cite[Theorem 1]{as77}.

\begin{lem}\label{gx.sbd}
Let $q\in[1,\fz)$, $\varphi$ be as in Definition \ref{d2.2} and $\vz\in\aa_q(\rn)$.
Then there exists a positive constant $C$ such that,
for all $\az\in[1,\fz)$, $t\in[0,\fz)$ and measurable functions $f$,
$$\int_{\mathbb R^n}\varphi(x,S_\az(f)(x))\,dx\le
C\az^{n(q-p/2)}\int_{\mathbb R^n}\varphi(x,S(f)(x))\,dx.$$
\end{lem}

\begin{proof}
For all $\lz\in(0,\fz)$, let $A_\lz:=\{x\in\rn:\ S(f)(x)>\lz\az^{n/2}\}$  and
$$U:=\{x\in\rn:\ M(\chi_{A_\lz})(x)>(4\az)^{-n}\},$$
where $M$ is the Hardy-Littlewood maximal function.
Since $\vz\in\aa_q(\rn)$, we see that
\begin{eqnarray}\label{az1}
\vz(U,\lz)&&=\vz\lf(\{x\in\rn:\ M(\chi_{A_\lz})(x)>(4\az)^{-n}\},\lz\r)\\
&&\ls (4\az)^{nq}\|\chi_{A_\lz}\|_{L_{\vz(\cdot,\lz)}^q(\rn)}^q
\sim \az^{nq}\vz(A_\lz,\lz)\noz
\end{eqnarray}
and, by \cite[Lemma 2]{as77}, we know that
\begin{eqnarray}\label{az2}
\az^{n(1-q)}\int_{U^\com}[S_\az(f)(x)]^2\vz(x,\lz)\,dx
\ls \int_{A_\lz^\com}[S(f)(x)]^2\vz(x,\lz)\,dx.
\end{eqnarray}
Thus, from \eqref{az1} and \eqref{az2}, it follows that
\begin{eqnarray*}
&&\vz\lf(\lf\{x\in\rn:\ S_\az(f)(x)>\lz\r\},\lz\r)\\
&&\hs\le \vz(U,\lz)+\vz\lf(U^\com\cap\{x\in\rn:\ S_\az(f)(x)>\lz\},\lz\r)\\
&&\hs\ls \az^{nq}\vz(A_\lz,\lz)+\lz^{-2}\int_{U^\com}[S_\az(f)(x)]^2\vz(x,\lz)\,dx\\
&&\hs\ls \az^{nq}\vz(A_\lz,\lz)+\az^{n(q-1)}\lz^{-2}\int_{A_\lz^\com}[S(f)(x)]^2\vz(x,\lz)\,dx\\
&&\hs\sim \az^{nq}\vz(A_\lz,\lz)+
\az^{n(q-1)}\lz^{-2}\int_0^{\lz\az^{n/2}}t \vz(\lf\{x\in\rn:\ S(f)(x)>t\r\},\lz)\,dt,
\end{eqnarray*}
which, together with the assumption that $\az\in[1,\fz)$, Lemma \ref{lem1}(ii),
the uniformly lower type $p$ and upper type $1$ properties of $\vz$,
further implies that
\begin{eqnarray*}
&&\int_{\mathbb R^n}\varphi(x,S_\az(f)(x))\,dx\\
&&\hs=\int_0^\fz\frac1\lz\vz\lf(\lf\{x\in\rn:\ S_\az(f)(x)>\lz\r\},\lz\r)\,d\lz\\
&&\hs\ls \az^{nq}\int_0^\fz\frac1\lz\vz(A_\lz,\lz)\,d\lz
+
\az^{n(q-1)}\int_0^\fz\lz^{-3}\int_0^{\lz\az^{n/2}}t \vz(\lf\{x\in\rn:\ S(f)(x)>t\r\},\lz)\,dt\,d\lz\\
&&\hs\ls \az^{n(q-p/2)}\int_0^\fz\frac1\lz\vz(\{x\in\rn:\ S(f)(x)>\lz\},\lz)\,d\lz\\
&&\hs\hs+
\az^{n(q-1)}\lf\{\int_0^\fz\lz^{-3}\int_0^{\lz}\lz \vz(\lf\{x\in\rn:\ S(f)(x)>t\r\},t)\,dt\,d\lz\r.\\
&&\hs\hs+\lf.
\int_0^\fz\lz^{-3}\int_\lz^{\lz\az^{n/2}}(\lz/t)^pt \vz(\lf\{x\in\rn:\ S(f)(x)>t\r\},t)\,dt\,d\lz\r\}\\
&&\hs\ls
\az^{n(q-p/2)}\int_\rn\vz(x,S(f)(x))\,dx\\
&&\hs\hs+
\az^{n(q-1)}\lf\{
\int_0^\fz\frac1t\lz \vz(\lf\{x\in\rn:\ S(f)(x)>t\r\},t)\,dt\r.\\
&&\hs\hs+\lf.
\int_0^\fz\frac1t\lf[\az^{(2-p)n/2}-1\r] \vz(\lf\{x\in\rn:\ S(f)(x)>t\r\},t)\,dt\r\}\\
&&\hs\ls\az^{n(q-p/2)}\int_{\mathbb R^n}\varphi(x,S(f)(x))\,dx.
\end{eqnarray*}
This finishes the proof of Lemma \ref{gx.sbd}.
\end{proof}

Using Lemma \ref{gx.sbd}, we obtain the following boundedness
of $g_\lambda^*$ from $\hv$ to $\lv$.

\begin{prop}\label{gast-bd}
Let $\vz$ be as in Definition \ref{d2.2}, $q\in [1,\fz)$, $\vz\in\aa_q(\rn)$,
and $\lambda\in({2q}/{p},\fz)$. Then, there exists a positive constant
$C_{(\varphi,q)}$ such that,
for all $f\in\hv$,
$$\|g_\lambda^*(f)\|_{L^\varphi(\rn)}\le C_{(\varphi,q)}\|f\|_\hv.$$
\end{prop}

\begin{proof}
For all $f\in\hv$ and $x\in\rn$,
\begin{eqnarray}\label{4.gx}
\lf[g_\lambda^*(f)(x)\r]^2&&=\int_{0}^{\infty}
 \int_{|x-y|<t}\lf(\frac{t}{t+|x-y|}\r)^{\lambda
 n}\lf|f\ast\phi_t(y)\r|^2\frac{dy\, dt}{t^{n+1}}\\
 &&\hs+\sum_{k=1}^{\infty}\int_{0}^{\infty}
 \int_{2^{k-1}t\le|x-y|<2^kt}\cdots\noz\\
&&\ls\lf[Sf(x)\r]^2+
 \sum_{k=1}^{\infty}2^{-kn(\lambda-1)}\lf[S_{ 2^k}f(x)\r]^2.\noz
\end{eqnarray}
Then from \eqref{4.gx},
Lemmas \ref{lem1}(i) and \ref{gx.sbd}, and $\lambda\in({2q}/{p},\fz)$, we deduce that
\begin{eqnarray*}
\int_\rn\vz(x,g_\lambda^*(f)(x))\,dx
&&\ls\sum_{k=0}^\fz\int_\rn\varphi\lf(x,2^{-kn\lz/2}S_{2^k}(f)(x)\r)\,dx\\
&&\ls \sum_{k=0}^\infty
 2^{-knp(\lambda-1)/2}2^{kn(q-p/2)}\int_\rn\varphi\lf(x,2^{-kn\lz/2}S_{2^k}(f)(x)\r)\,dx\\
&&\ls\int_\rn\varphi(x,S(f)(x))\,dx.
\end{eqnarray*}
By Lemma \ref{lem2}(i), we see that
\begin{eqnarray*}
\int_\rn\vz\lf(x,\frac{g_\lambda^*(f)(x)}{\|f\|_\hv}\r)\,dx
&&\ls\int_\rn\varphi\lf(x,\frac{S(f)(x)}{\|f\|_\hv}\r)\,dx\\
&&\sim\int_\rn\varphi\lf(x,\frac{S(f)(x)}{\|S(f)\|_\lv}\r)\,dx\sim1,
\end{eqnarray*}
which, together with Lemma \ref{lem3}(i), then completes the proof of
Proposition \ref{gast-bd}.
\end{proof}

By Propositions \ref{van}, \ref{prop3} and \ref{gast-bd},
we have the following $g_\lz^*$-function characterization
of $\hv$. We omit the details.

\begin{thm}\label{gast-ch}
Let $\vz$ be as in Definition \ref{d2.2}, $q\in[1,\fz)$, $\vz\in\aa_q(\rn)$
and $\lambda\in(2q/p,\fz)$.
Then $f\in\hv$ if and only if $f\in\cs'(\rn)$, $f$ vanishes weakly at infinity and $g_\lz^*(f)\in\lv$
and, moreover,
$$\frac1C\|g_\lz^*(f)\|_\lv\le\|f\|_\hv\le C\|g_\lz^*(f)\|_\lv$$
with $C$ being a positive constant independent of $f$.
\end{thm}

We point out that the range of $\lz$
in Theorem \ref{gast-ch} is the known best possible,
even when $\vz(x,t):=t^p$ for all $x\in\rn$ and $t\in(0,\fz)$,
or $\vz(x,t):=w(x)t^p$ for all $x\in\rn$ and $t\in(0,\fz)$,
with $q\in[1,\fz)$ and $w\in A_q(\rn)$; see, respectively,
\cite[p.\,221,\,Corollary (7.4)]{fs82} and \cite[Theorem 2]{as77}.


\medskip

\noindent \emph{Yiyu Liang} $\&$ \emph{Dachun Yang} (Corresponding author)

\smallskip

\noindent School of Mathematical Sciences, Beijing Normal University,
Laboratory of Mathematics and Complex Systems, Ministry of
Education, Beijing 100875, People's Republic of China

\smallskip

\noindent E-mails: \emph{yyliang@mail.bnu.edu.cn} $\&$ \emph{dcyang@bnu.edu.cn}

\medskip

\noindent \emph{Jizheng Huang}

\smallskip

\noindent College of Sciences, North
China University of
Technology, Beijing 100144, People's Republic of China

\smallskip

\noindent E-mail: \emph{hjzheng@163.com}

\end{document}